\documentclass{amsart}

\usepackage{amsfonts, amssymb}

\newtheorem{theorem}{Theorem}[section]
\newtheorem{lemma}[theorem]{Lemma}
\newtheorem{proposition}[theorem]{Proposition}
\newtheorem{corollary}[theorem]{Corollary}

\theoremstyle{definition}
\newtheorem{definition}[theorem]{Definition}
\newtheorem{example}[theorem]{Example}

\theoremstyle{remark}
\newtheorem{remark}[theorem]{Remark}

\numberwithin{equation}{section}

\newcommand{\N}{\mathbb{N}}                   % natural numbers
\newcommand{\R}{\mathbb{R}}                   % real numbers
\newcommand{\C}{\mathbb{C}}                   % complex numbers
\newcommand{\K}{\mathbb{K}}                   % field K
\newcommand{\vp}{\varphi}                     % mapping phi
\newcommand{\nR}{\mathrm{NR}}                 % Gabrielov irregular
\newcommand{\dl}{\mathfrak{d}}                % totally real embedding
\newcommand{\Dl}{\mathfrak{D}}                % totally real diagonal
\newcommand{\OO}{\mathcal{O}}                 % local ring
\newcommand{\fbd}{\mathrm{fbd}}               % fibre dimension   
\newcommand{\reg}{\mathrm{reg}}               % regular
\newcommand{\sng}{\mathrm{sng}}               % singular
\newcommand{\II}{\mathfrak{I}}                % ideal I
\def\O{\Omega}

\begin{document}

\title{On the holomorphic closure dimension of real analytic sets}

\author{Janusz Adamus}
\address{J. Adamus, Department of Mathematics, The University of Western Ontario, London, Ontario N6A 5B7 Canada,
         and Institute of Mathematics, Jagiellonian University, ul. {\L}ojasiewicza 6, 30-348 Krak{\'o}w, Poland}
\email{jadamus@uwo.ca}

\author{Rasul Shafikov}
\address{R. Shafikov, Department of Mathematics, The University of Western Ontario, London, Ontario N6A 5B7 Canada}
\email{shafikov@uwo.ca}
\thanks{Research was partially supported by the Natural Sciences and Engineering Research Council of Canada.}

\subjclass[2000]{Primary 32B20, 32V40}

\keywords{real analytic sets; semianalytic sets; holomorphic closure dimension; complexification; Gabrielov regularity; CR dimension}

\begin{abstract}
Given a real analytic (or, more generally, semianalytic) set $R$ in $\C^n$ (viewed as $\R^{2n}$), there is, for every $p\in\bar{R}$, 
a unique smallest complex analytic germ $X_p$ that contains the germ $R_p$. We call $\dim_{\C}X_p$ the \emph{holomorphic closure dimension} 
of $R$ at $p$. We show that the holomorphic closure dimension of an irreducible $R$ is constant on the complement of a closed proper 
analytic subset of $R$, and discuss the relationship between this dimension and the CR dimension of $R$.
\end{abstract}

\maketitle

%%%%%%%%%%%%%%%%%%%%%%%%%%%%%%%%%%%%%%%%%%%%%%%%%%
\section{Main Results}
\label{sec:intro}

Given a real analytic set $R$ in $\C^n$ (we may identify $\C^n$ with $\R^{2n}$), we consider the germ $R_p$ at a point $p\in R$, and define $X_p$ to be the unique smallest complex analytic germ at $p$ that contains $R_p$. We will call $\dim_{\C}X_p$ the \emph{holomorphic closure dimension} of $R$ at $p$ (denoted by $\dim_{HC}R_p$). It is natural to ask how this dimension varies with $p\in R$. In~\cite[Thm.~1.1]{S} the second author showed that, if $R$ is irreducible of pure dimension, then $\dim_{HC}R_p$ is constant on a dense open subset of $R$, which allows us to speak of the {\it generic} holomorphic closure dimension of $R$ in this case. Upper semicontinuity implies that the generic value of $\dim_{HC}R_p$ is the smallest value of the holomorphic closure dimension on $R$. It remained an open problem whether the jump in the holomorphic closure dimension actually occurs. In the present paper we construct real analytic sets with non-constant holomorphic closure dimension, and study the structure of the locus of points where this dimension is not generic.

\begin{theorem}
\label{thm:holo-dim}
Let $R$ be an irreducible real analytic set in $\C^n$, of dimension $d>0$. Then there exists a closed real analytic subset $S\subset R$ of dimension less than $d$, such that the holomorphic closure dimension of $R$ is constant on $R\setminus S$.
\end{theorem}

We note that the holomorphic closure dimension is well defined on all of $R$, including its
singular locus. If $R$ is not of pure dimension, then (see, e.g., \cite{C}) the locus of smaller dimension is contained in a proper real analytic subset of dimension less than $d$, which can be included into $S$. In particular, $S$ may have non-empty interior in $R$ (see Example~\ref{ex:cartan}). We also note that unless additional conditions are imposed on $R$, the germs $X_p$ realizing the holomorphic closure
dimension of $R$ at $p$ cannot be glued together to form a global complex analytic set containing $R$.

It should be stressed that the discontinuity in the holomorphic closure dimension is by no means related to (real) singularities of $R$. 
As shown in Section~\ref{sec:examples}, the exceptional set $S$ may be non-empty even in the case when $R$ is a smooth real analytic 
submanifold of $\C^n$. Instead, the holomorphic closure dimension should be viewed as a measure of how badly a real analytic set is 
locally twisted  with respect to the complex structure of its ambient space. Thus, one can interpret this dimension as a generalization 
of the concept of a uniqueness set. Indeed, a real analytic set $R$ in $\C^n$ is a local uniqueness set, at a point $p\in R$, for holomorphic 
functions in $\C^n$ if and only if $\dim_{HC} R_p=n$.

The notion of the holomorphic closure dimension can be extended to the class of semianalytic sets. Recall that a set $R\subset \R^{n}$ is semianalytic if, for every point $p\in \R^n$, there exist a neighbourhood $U$ and a finite collection of functions $f_{jk}$ and $g_{jk}$ real analytic in $U$ such that  
\begin{equation*}
  R\cap U =\bigcup_{j} \{x\in U: f_{jk}(x)=0,\ g_{jk}(x)>0,\ k=1,\dots,l\}. 
\end{equation*}
If $R$ is a semianalytic set in $\C^n\cong\R^{2n}$, of dimension $d>0$ at a point $p\in\bar{R}$, then there are an open neighbourhood $U$ of $p$, and a real analytic set $Z\subset U$ of dimension $d$, such that $R\cap U\subset Z$ and $R\setminus\{$\emph{interior of} $R$ \emph{in} $Z\}$ is semianalytic of dimension at most $d-1$ (see~\cite[Thm.~2.13]{BM}). Then we may define the holomorphic closure dimension of $R$ at $p$ as $\dim_{HC}Z_p$. Theorem~\ref{thm:holo-dim} implies then a similar statement for semianalytic sets.

\begin{corollary}
\label{cor:semi-holo-dim}
Let $R$ be a semianalytic set of dimension $d>0$ in $\C^n$. Suppose one of the following conditions holds:
\begin{itemize}
\item[(a)] $R$ is contained in an irreducible real analytic set of the same dimension;
\item[(b)] $R$ is connected, pure-dimensional and is locally contained in a locally irreducible real analytic set of the same dimension. 
\end{itemize}
Then there exists a closed semianalytic subset $S\subset R$, of dimension less than $d$, and such that the holomorphic closure dimension of $R$ is constant on $R\setminus S$.
\end{corollary}

Curiously enough, the discontinuity of the holomorphic closure dimension is a purely transcendental phenomenon. This is a direct
consequence of the fundamental connection between the holomorphic closure dimension and the Gabrielov irregularity (see below),
which is established in the proof of Theorem~\ref{thm:holo-dim}.

\begin{theorem}
\label{thm:semialgebraic}
Let $R$ be a connected semianalytic set of pure dimension $d>0$ in $\C^n$, contained (resp. locally contained) in an irreducible (resp. locally irreducible) real analytic set of the same dimension. Suppose further that $R$ is semialgebraic, or that $n<3$. 
Then the holomorphic closure dimension is constant on $R$.
\end{theorem}

The holomorphic closure dimension is a natural biholomorphic invariant associated with the germ at $p$ of a real analytic set $R$. 
It follows from Theorem~\ref{thm:semialgebraic} that if $R$ has a jump in the holomorphic closure dimension at a point $p\in R$, 
then $R$ is not \emph{algebraizable} at this point; i.e., there does not exist a local biholomorphic change of coordinates sending 
the germ of $R$ at $p$ into the germ of an algebraic variety of the same dimension.
\par

If $R$ is a CR manifold, then the holomorphic closure dimension is constant on $R$, and its value is uniquely determined by the CR dimension of $R$.

\begin{proposition}[cf.{\cite[Lemma~9.3]{hl}}]
\label{prop:cr}
Suppose $M$ is a real analytic CR manifold of dimension $d$ and CR dimension $m\ge 0$. 
Then for any $p\in M$, $\dim_{HC} M_p=d-m$. 
\end{proposition}

The above proposition is well-known. The proof that we provide here for the sake of completeness essentially repeats the argument 
of \cite[Lemma~9.3]{hl}, where the result is proved for CR manifolds $M$ of hypersurface type (i.e., $\dim_{CR}M=(d-1)/2$). On the 
other hand, if $M$ is a real analytic manifold which is not CR, then at CR singularities, the number $d-m$ can be different 
from the holomorphic closure dimension of $M$. In fact, at CR singularities of $M$ there does not seem to be any relationship between 
the two invariants. Indeed, the holomorphic closure dimension in general is bounded only by the dimension of the ambient space (see Remark~\ref{rem:sky's-the-limit}). 

\begin{theorem}
\label{cor:cr}
Let $R$ be an irreducible real analytic set of pure dimension $d$, and generic holomorphic closure dimension $h$. Then there exists a 
semianalytic subset $Y$ of $R$, $\dim Y<d$, such that $R\setminus Y$ is a CR manifold of CR~dimension $m=d-h$.
\end{theorem}

In particular, the above theorem implies that similar to real analytic manifolds, the {\it generic} CR dimension is correctly defined 
also for real analytic sets of pure dimension.

Since most of the arguments involved in the proofs are local, all of the above results can be
easily extended to the case when $R$ is a real analytic or semianalytic subset of an arbitrary 
connected complex manifold. For the sake of simplicity of the statements we chose to work in $\C^n$.

In the next section we give some background material. Theorem~\ref{thm:holo-dim} is proved in Section~\ref{sec:proof-main}, and the 
proofs of Theorem \ref{thm:semialgebraic} and Corollary~\ref{cor:semi-holo-dim} are given in Section~\ref{sec:corollaries}. Proposition~\ref{prop:cr} and Theorem~\ref{cor:cr} are proved in Section~\ref{prop-proofs}. Finally, the last section contains relevant explicit examples.

%%%%%%%%%%%%%%%%%%%%%%%%%%%%%%%%%%%%%%%%%%%%%%%%%%
\medskip
%%%%%%%%%%%%%%%%%%%%%%%%%%%%%%%%%%%%%%%%%%%%%%%%%%
\section{Background}\label{sec:background}

Given a real linear subspace $L$ in $\C^n$ of dimension $d$, we define the {\it CR dimension} of $L$ to be the largest $m$ 
such that $L$ contains a complex linear subspace of (complex) dimension $m$. Clearly, $0\le m\le \left[\frac{d}{2}\right]$. 
A real submanifold $M$ in $\C^n$ of real dimension $d$ is called a {\it CR manifold} of CR dimension $m$, if the tangent 
space $T_p M$ contains a complex linear subspace of dimension $m$, where $m$ is independent of the point $p\in M$. We write 
$\dim_{CR}M=m$. In particular, if $m=0$, then $M$ is called a {\it totally real} submanifold.

A real (resp. complex) analytic set $X$ in an open set $\O\subset\R^n$ (resp. $\O\subset\C^n$) is a closed subset of $\O$, locally (i.e., in a neighbourhood of each point in $\O$) defined as the zero locus of finitely many real analytic (resp. holomorphic) functions. $X$ is called 
irreducible (in $\O$) if it cannot be represented as a union of two non-empty real (resp. complex) analytic sets in $\O$ properly contained in $X$. Any germ of a real or complex analytic set $X$ admits decomposition into a finite union of irreducible germs, while a global decomposition 
into irreducible components in general exists only for complex analytic sets (see, e.g., \cite{n}).

Throughout this article, the dimension of a set $X$ is understood in the following sense: if $X$ is a subset of a $\K$-manifold $M$ ($\K=\R$ or $\C$), then
\[
\dim_{\K} X =\max\{\dim_{\K}N :\ N\subset X	,\ N\mathrm{\ a\ closed\ submanifold\ of\ an\ open\ subset\ of\ }M\}\,.
\]
The dimension of the germ $X_p$ of a set $X$ at a point $p\in M$ is then defined as
\[
\dim_{\K} X_p = \min\{\dim_{\K}(X\cap U) :\ U\mathrm{\ an\ open\ neighbourhood\ of\ } p \mathrm{\ in\ }M\}\,.
\]

Let $p$ be a point in $\R^n$ viewed as as subset of $\C^n=\R^n+i\R^n$. For the germ $R_p$ of a real analytic set 
$R\subset\R^n$, there is a unique germ $R_p^c$ of a complex analytic set at $p$, such that $R_p\subset R_p^c$ 
and any germ of a holomorphic function at $p$ which vanishes on $R_p$, vanishes also on $R_p^c$ (see, e.g., \cite{C} 
or \cite{n}). $R_p^c$ is called the \emph{complexification} of $R_p$. If $\II(R_p)\subset\OO^{\R}_{p}$ (resp. 
$\II(R_p^c)\subset\OO^{\C}_{p}$) is the ideal of real analytic (resp. holomorphic) germs vanishing on $R_p$ (resp. 
on $R_p^c$), then $\II(R_p^c)=\II(R_p)\otimes_{\R}\C$. It follows that $\dim_{\R}R_p=\dim_{\C}R_p^c$ 
($=2\!\cdot\!\dim_{\R}R_p^c$), and if $R_p$ is irreducible, then so is~$R_p^c$.

For our purposes, it is convenient to realize complexification by means of the following construction. Let 
$\dl:\C^n_{\zeta}\to\C^{2n}_{(z,w)}$ be the map defined by $\dl(\zeta)=(\zeta,\bar\zeta)$. Then 
$\Dl=\dl(\C^n)$ is a totally real embedding of $\C^n$ into $\C^{2n}$. Suppose $R$ is a real analytic set in 
$\C^n$, and $p\in R$. Then the complexification $R^c_p$ of $R_p$ can be identified with the complexification of the 
germ of $\dl(R)$ at $\dl(p)$ in $\C^{2n}$, that is, the smallest germ of a complex analytic set in $\C^{2n}$ 
which contains the germ of $\dl(R)$ at~$\dl(p)$. 

Let now $X$ be a complex analytic set in an open neighbourhood $U$ of $p$ in $\C^n$, defined by 
$g_1,\dots,g_t\in\OO(U)$. We set
\begin{equation}
\label{eq:xz}
\begin{split}
 X_z &=\{(z,w)\in U': g_k(z)=\sum_{|\nu|\ge 0} c_{k\nu}z^{\nu}=0,\; k=1,\dots,t\}\\
 X_w &=\{(z,w)\in U': \bar{g}_k(w)=\sum_{|\nu|\ge 0} \overline{c}_{k\nu} w^{\nu}=0,\; k=1,\dots,t\}\,,
\end{split}
\end{equation}
where $U'$ is some small open neighbourhood of $\dl(p)$ in $\C^{2n}$. Let $\pi_z:\C^{2n}_{(z,w)}\to\C^n_z$ 
and $\pi_w:\C^{2n}_{(z,w)}\to\C^n_w$ be the coordinate projections. Then $X_z=\pi_z(X_z)\times\C^n$ and 
$X_w=\C^n\times\pi_w(X_w)$, as the defining equations of $X_z$ (resp. $X_w$) do not involve variables $w$ 
(resp. $z$). Therefore, the set
\begin{equation}
\label{eq:x-hat}
\widehat X:= X_z\cap X_w=\pi_z(X_z)\times\pi_w(X_w)
\end{equation}
is complex analytic (in $U'$) of dimension equal twice the dimension of $X$. If $X_p$ is irreducible, then the 
complexification $X^c_p$ of $X_p$ (viewed as a real analytic germ) coincides with $\widehat{X}_{\dl(p)}$.
Indeed, clearly $\dl(X)\subset\widehat{X}$, hence $X^c_p\subset\widehat{X}_{\dl(p)}$. But the irreducibility of $X_p$ implies that of $\widehat{X}_{\dl(p)}$ (by~\eqref{eq:xz} and~\eqref{eq:x-hat}), and $\dim X^c_p=2\dim X_p=\dim\widehat{X}_{\dl(p)}$, so $X^c_p=\widehat{X}_{\dl(p)}$.
\medskip

A subset $R$ of $\R^n$ is called {\it semianalytic} if for any point $p\in \R^n$ there exist a neighbourhood $U$ 
and a finite number of functions $f_{jk}$ and $g_{jk}$ real analytic in $U$ such that  
\begin{equation*}
  R\cap U =\bigcup_{j} \{\zeta\in U: f_{jk}(\zeta)=0,\ g_{jk}(\zeta)>0,\
  k=1,\dots,l\}. 
\end{equation*}
In particular, a real analytic set is semianalytic. A point $p$ of a semianalytic set $R$ is called \emph{regular} if near $p$ the 
set $R$ is just a real analytic manifold. The complement of the set $R^{\reg}$ of regular points is called the \emph{singular} locus of $R$, and will be denoted by $R^{\sng}$. The singular locus of a semianalytic set is itself a (closed) semianalytic set (\cite{L2}). For more details on real analytic sets see, e.g., \cite{C} or \cite{n}. For an exposition of semianalytic sets we refer the reader to~\cite{BM} (and its bibliography), since the classical monograph~\cite{L2} is difficult to access.
\medskip

The other main technique used in the proof of Theorem~\ref{thm:holo-dim} is Gabrielov regularity, which we will now briefly recall.
The notion of regularity, introduced in the seminal paper~\cite{Ga} of Gabrielov, plays an important role in subanalytic geometry, 
where it is responsible for certain ``tameness'' properties of subanalytic sets (see, e.g.,~\cite{BM2}, or the recent~\cite{ABM} 
for details). For our purposes though, it suffices to consider regularity in the complex analytic context. Let $\vp:M\to N$ be a 
holomorphic mapping of complex manifolds, and let $a\in M$. Then $\vp$ induces a pull-back homomorphism of local rings 
$\vp^*_a:\OO_{N,\vp(a)}\to\OO_{M,a}$, and further, a homomorphism of their completions $\hat\vp^*_a:\hat\OO_{N,\vp(a)}\to\hat\OO_{M,a}$. 
Consider the following three kinds of rank of $\vp$ at $a$:
\begin{align}
\notag
r^1_a(\vp)=& \mathrm{\ the\ generic\ rank\ of}\ \vp\mathrm{\ near}\ a\,,\\
\notag
r^2_a(\vp)=& \mathrm{\ the\ Krull\ dimension\ of}\ \ \hat{\mathcal{O}}_{N,\vp(a)}/\ker\hat{\vp}^*_a\,,\\
\notag
r^3_a(\vp)=& \mathrm{\ the\ Krull\ dimension\ of}\ \ \mathcal{O}_{N,\vp(a)}/\ker\vp^*_a\,.
\end{align}
It is easy to see that $r^1_a(\vp)\leq r^2_a(\vp)\leq r^3_a(\vp)$ (both inequalities can be strict, see Example~\ref{ex:no-semi-coherent}). 
Gabrielov~\cite{Ga} proved that if $r^1_a(\vp)=r^2_a(\vp)$ then $r^2_a(\vp)=r^3_a(\vp)$.

\begin{definition}
\label{def:G-regular}
The map $\vp$ is called \emph{(Gabrielov) regular} at $a$ if $r^1_a(\vp)=r^3_a(\vp)$. Equivalently, there exists an open neighbourhood $U$ of 
$a$ in $M$ such that the image $\vp(U)$ is contained in a locally analytic subset of $N$ of dimension $r^1_a(\vp)$.
\end{definition}

Let $\nR(\vp)$ denote the locus in $M$ of points at which $\vp$ is not regular. We have:

\begin{theorem}[{\cite[Thm.~1]{P2}}]
\label{thm:Pawlucki}
The non-regular locus $\nR(\vp)$ is a nowhere dense complex analytic subset of $M$.
\end{theorem}

Let now $A$ be an irreducible complex analytic subset of $M$. We will say that $\vp|_A$ is (Gabrielov) regular at a point $a\in A$ if there exists an open neighbourhood $U$ of $a$ such that $\dim_{\vp(a)}\vp(A\cap U)=\dim_{\vp(a)}\overline{\vp(A\cap U)}^{\mathrm{Zar}}$, where 
$\overline{\vp(W)}^{\mathrm{Zar}}$ denotes the Zariski closure, that is, the smallest locally analytic subset of $N$ containing $\vp(W)$. 
Let again $\nR(\vp|_A)$ denote the non-regular locus of $\vp|_A$. It is well known that, by composing with desingularization of $A$, Theorem~\ref{thm:Pawlucki} implies that $\nR(\vp|_A)$ is a nowhere dense analytic subset of $A$. We include the following lemma for completeness.

\begin{lemma}
\label{lem:non-regular}
The non-regular locus $\nR(\vp|_A)$ is a nowhere dense analytic subset of $A$.
\end{lemma}

\begin{proof}
For any relatively compact open set $\O\subset M$, there exist a complex manifold $\tilde{A}$ and a proper modification $\sigma:\tilde{A}\to\O\cap A$ centred at the singular locus of $\O\cap A$ (that is, $\sigma$ is a proper analytic mapping, $\dim\tilde{A}=\dim A$, $\sigma^{-1}(\O\cap A^{\sng})$ is nowhere dense in $\tilde{A}$, and $\sigma$ restricted to $\tilde{A}\setminus\sigma^{-1}(\O\cap A^{\sng})$ is an isomorphism); see, e.g., \cite{BM3}.

Fix $a\in\O\cap A$. Suppose first that the germ $A_a$ is reducible. Then, after shrinking $\O$, we can assume that $\O\cap A=A_1\cup\dots\cup A_s$ is a decomposition into irreducible analytic sets, each irreducible at $a$. Then, for an open neighbourhood $U$ of $a$ in $\O$, we have $\dim_{\vp(a)}\overline{\vp(U\cap A)}^{\mathrm{Zar}}=\max_{j=1,\dots,s}\dim_{\vp(a)}\overline{\vp(U\cap A_j)}^{\mathrm{Zar}}$, and hence $\vp|_A$ is regular at $a$ if and only if $\vp|_{A_j}$ is regular at $a$ for all $j$. Moreover, for any $j=1,\dots,s$, the inverse image $\sigma^{-1}(A_j)$ is an analytic subset of $\tilde{A}=\sigma^{-1}(\O\cap A)$, of dimension $\dim\tilde{A}$, and hence it is a union of some connected components of $\tilde{A}$.
Thus, $\sigma^{-1}(A_j)$ is a neighbourhood of $\tilde{a}$, for every $\tilde{a}\in\sigma^{-1}(a)$ to which it is adherent, and it suffices to consider the case when $A_a$ is irreducible.\par

Since the generic fibres of $\sigma$ are zero-dimensional, the generic rank of $\vp$ near $a$ equals that of $\vp\circ\sigma$ near $\tilde{a}$ for every $\tilde{a}\in\sigma^{-1}(a)$ (as they are both equal to the difference between $\dim A$ and the generic fibre dimension of $\vp$, by~\cite[Ch.\,V,\S\,3.3]{Lo}). Hence $r^1_{\tilde{a}}(\vp|_{\O\cap A}\circ\sigma)=r^1_a(\vp|_{\O\cap A})$ for every $\tilde{a}\in\sigma^{-1}(a)$.
On the other hand, for every $\tilde{a}\in\sigma^{-1}(a)$, we have $\ker\sigma^*_{\tilde{a}}=(0)$ (by irreducibility of $A_a$), and hence $\ker\vp^*_a=\ker(\vp\circ\sigma)^*_{\tilde{a}}$. It follows that $r^3_{\tilde{a}}(\vp|_{\O\cap A}\circ\sigma)=r^3_a(\vp|_{\O\cap A})$ for all $\tilde{a}\in\sigma^{-1}(a)$, and consequently, that $\vp|_{\O\cap A}$ is regular at $a$ if and only if $\vp|_{\O\cap A}\circ\sigma$ is regular at $\tilde{a}$ for every $\tilde{a}\in\sigma^{-1}(a)$. Therefore $\nR(\vp|_{\O\cap A})=\sigma(\nR(\vp|_{\O\cap A}\circ\sigma))$, which is a nowhere dense analytic subset of $\O\cap A$, by properness of $\sigma$.
\end{proof}
\smallskip

For a holomorphic mapping $\vp:A\to B$ of complex analytic sets, we will denote by $\fbd_a\vp$ the dimension of the germ at $a\in A$ of the 
fibre $\vp^{-1}(\vp(a))$. If $A$ is irreducible, the minimal fibre dimension is attained on a dense Zariski open subset of $A$, and is then 
called the generic fibre dimension of $\vp$, and denoted by $\lambda(\vp)$ (see, e.g.,~\cite[Ch.\,V,\S\,3.3]{Lo}).

%%%%%%%%%%%%%%%%%%%%%%%%%%%%%%%%%%%%%%%%%%%%%%%%%%
\medskip
%%%%%%%%%%%%%%%%%%%%%%%%%%%%%%%%%%%%%%%%%%%%%%%%%%

\section{Proof of Theorem~\ref{thm:holo-dim}}
\label{sec:proof-main}

Let $R$ be an irreducible real analytic set in $\C^n$ and let $p\in R$. Suppose first that $R$ is irreducible at~$p$. Put $d=\dim_{\R}R_p$. Then, by~\cite[Prop.~10]{C}, there exists a real analytic subset $R'\subset R\cap\Omega$ in an open neighbourhood $\Omega$ of $p$, such that $R'$ is of 
dimension at most $d-1$, and every point of $R\setminus R'$ lies in the closure of a $d$-dimensional real analytic manifold contained in $R$.\par

Let $X$ be a complex analytic set in an open neighbourhood $U\subset\Omega$ of $p$ in $\C^n$ such that 
$R_p\subset X_p$, and $X_p$ is the smallest germ with this property. Let $X_z$ and $X_w$ be defined as in 
\eqref{eq:xz}. Then, with the notation of Section~\ref{sec:background},  for $q:=\dl(p)\in\C^{2n}$,
\[
\widehat{X}_q=(X_z\cap X_w)_q
\]
is a complex analytic germ containing $R^c_p$, of dimension equal twice that of $X_p$.
Let $A$ be a complex analytic representative of the complexification $R^c_p$ at $q$; i.e., $A_q=R^c_p$. Then
\[
R_p\subset X_p\ \Rightarrow\ A_q\subset \widehat{X}_q\ \Rightarrow\ (\pi_z(A))_{\pi_z(q)}\subset(\pi_z(X_z))_{\pi_z(q)}.
\]
Hence, $\dim (\pi_z(A))_{\pi_z(q)} \le \dim X_p$, by~\eqref{eq:xz} and~\eqref{eq:x-hat}.

On the other hand, suppose $(\pi_z(A))_{\pi_z(q)}\subset\tilde{Z}_{\pi_z(q)}$ for some complex analytic $\tilde{Z}$ 
in a neighbourhood $V$ of $\pi_z(q)$ in $\C^n$, and $\dim \tilde Z < \dim X_p$. Say, 
$\tilde{Z}=\{z\in V: g_k(z)=0, k=1,\dots,t\}$. Define a complex analytic set $Z$ in a neighbourhood $U$ of 
$p$ as $Z=\{\zeta\in U: g_k(\zeta)=0, k=1,\dots,t\}$. Then $Z=\dl^{-1}(\pi^{-1}_z(\tilde{Z})\cap\Dl)$, and hence
\[
R_p\subset(\dl^{-1}(A\cap\Dl))_p\subset(\dl^{-1}(\pi_z^{-1}(\tilde{Z})\cap\Dl))_p=Z_p\,,
\]
where $\dim_{\C}Z_p=\dim_{\C}\tilde{Z}_{\pi_z(q)}$. This, however, contradicts the choice of $X_p$.
Therefore,
\begin{multline}
\tag{\dag}
R_p\subset X_p \ \ \mathrm{and}\ \ \dim_{\C}X_p=r\quad\Longleftrightarrow\quad \pi_z\mathrm{\ maps\ the\ germ}\ R^c_p\ \mathrm{into\ a\ complex}\\
\mathrm{analytic\ subgerm\ of}\ (\C^n)_{\pi_z(q)}\mathrm{\ of\ dimension}\ r\,.
\end{multline}
\smallskip

Let again $A$ be an irreducible complex analytic set in a neighbourhood $U'$ of $q\in\C^{2n}$, such that $A_q=R^c_p$, and let a proper 
real analytic subset $R'\subset R\cap\Omega$ be as at the beginning of the proof. Then, for every $x\in R\setminus R'$ in a sufficiently 
small open neighbourhood $U\subset\Omega$ of $p$, the germ $A_{\dl(x)}$ contains the complexification $R^c_x$, and
\[
\dim_{\C}R^c_x=\dim_{\C}A_{\dl(x)}=\dim_{\C}A_q=d\,.
\]
Consider the mapping $\vp:A\to\C^n$ induced by the restriction $\pi_z|_{U'}$. By irreducibility of $A$, every point $\eta\in A$ has arbitrarily small open neighbourhoods $V^{\eta}$ with the property
\[
\tag{\ddag}
\dim_{\C}\vp(V^{\eta})=\dim_{\C}A-\lambda(\vp)=d-\lambda(\vp)\,,
\]
where $\lambda(\vp)$ is the generic fibre dimension of $\vp$ on $A$ (see, e.g.,~\cite[Ch.\,V,\S\,3]{Lo}). Let $B=\nR(\vp)$ be the non-regularity locus of $\vp$. By Lemma~\ref{lem:non-regular}, $B$ is a proper analytic subset of $A$.\par

Now, for every $\eta\in A\setminus B$, the germ $A_{\eta}$ is mapped by $\vp$ into an analytic subgerm of $\C^n_{\vp(\eta)}$ of dimension $d-\lambda(\vp)$, whilst, for every $\eta\in B$, and arbitrarily small open neighbourhood $V^{\eta}$ of $\eta$,
\[
\vp(V^{\eta})\subset Z, \ \mathrm{and}\ Z\mathrm{\ locally\ analytic} \ \Rightarrow\ \dim_{\C}Z_{\vp(\eta)}>d-\lambda(\vp)\,.
\]
Put $S=(R'\cup\dl^{-1}(B\cap\Dl))\cap U$. Then $S$ is a proper real analytic subset of $R\cap U$, and the holomorphic closure dimension of $R$ is equal to $d-\lambda(\vp)$ on $(R\cap U)\setminus S$, by $(\dag)$ and $(\ddag)$.
\smallskip

To complete the proof, it remains to show that, if $R$ is reducible at $p$, there exist an open neighbourhood $\Omega$ of $p$, and real analytic subsets $R_j\subset R\cap\Omega$, $j=1,\dots,s$, each irreducible at $p$, such that $R\cap\Omega=R_1\cup\dots\cup R_s$, and if the holomorphic closure dimension of the $R_j$ is constantly $h_j$, say, on the complement of a proper analytic subset $S_j\subset R_j$, then $h_j=h_k$ for all $j,k=1,\dots,s$. For this it suffices to know that, given any two points in $R$, there exists a path $\gamma\subset R$ that connects these points, and has the property that if $a\in\gamma$ is a point at which $R$ is reducible, then $\gamma$ stays in the same local irreducible component of $R$ at $a$. For a proof by contradiction, let $\Sigma$ denote the locus of those points of $R$ that can be connected with a given point $q$ by a path $\gamma$ which satisfies the above property, and suppose that $\Sigma\neq R$. We claim that $\Sigma$ is a real analytic set. Indeed, every point $a\in\Sigma$ admits an open neighbourhood $U$ such that $\Sigma\cap U$ is the union of some of the local irreducible components of $R\cap U$, and thus $\Sigma\cap U$ is real analytic in $U$. Since $\Sigma$ is clearly closed, it is a real analytic subset of $R$. Let now $\Sigma'=\overline{R\setminus\Sigma}$. Then $\Sigma'$ is also real analytic. Indeed, if $a\in R\setminus\Sigma$, then an open neighbourhood of $a$ in $R$ is contained in $\Sigma'$, and if $a\in(\overline{R\setminus\Sigma})\setminus(R\setminus\Sigma)$, then $R$ is reducible at $a$, and in a small open neighbourhood $U$ of $a$, $\Sigma'\cap U$ is the union of those local irreducible components of $R$ at $a$ that are not in $\Sigma\cap U$. Thus $R=\Sigma\cup\Sigma'$ is reducible; a contradiction. Therefore $\Sigma=R$, which completes the proof.

%%%%%%%%%%%%%%%%%%%%%%%%%%%%%%%%%%%%%%%%%%%%%%%%%%
\medskip
%%%%%%%%%%%%%%%%%%%%%%%%%%%%%%%%%%%%%%%%%%%%%%%%%%

\section{Semianalytic consequences}
\label{sec:corollaries}

\begin{proof}[Proof of Corollary~\ref{cor:semi-holo-dim}]
For the proof of case \emph{(a)}, let $R$ be a semianalytic set of dimension $d>0$, contained in an irreducible real analytic $d$-dimensional set $Z$ in $\C^n$. By Theorem~\ref{thm:holo-dim}, there exists a real analytic set $S'\subset Z$ of dimension at most $d-1$, and such that the holomorphic closure dimension is constant on $Z\setminus S'$; say, equal to $h_Z$. By the semianalytic stratification, there is a closed semianalytic subset $R'\subset R$, such that $\dim_{\R}R'\leq d-1$, and $\dim_{\R}R_p=d$ for every point $p\in R\setminus R'$. Then $S=S'\cup R'$ is a closed semianalytic subset of $R$, of dimension at most $d-1$, and such that, for every $p\in R\setminus S$, $\dim_{HC}R_p=\dim_{HC}Z_p=h_Z$\,.\par

Now, for \emph{(b)}, suppose that $R$ is connected, semianalytic of pure dimension $d>0$, and that, for every point $p\in R$, there exist an open neighbourhood $U^p$ and an irreducible real analytic set $Z\subset U^p$ of dimension $d$, such that $R\cap U^p\subset Z$. Then, as above, there is a closed semianalytic subset $S\subset R\cap U^p$, of dimension at most $d-1$, such that the holomorphic closure dimension is constant on $(R\cap U^p)\setminus S$; say, equal to $h_p$. It remains to show that, for any two points $p$ and $q$ in $R$, $h_p=h_q$. Let then $\gamma$ be a path joining $p$ and $q$ in $R$, and let, for every $x\in\gamma$, $U^x$ be an open neighbourhood of $x$ with the above properties. By compactness of $\gamma$, there are finitely many $p=x_0,x_1,\dots,x_s=q$ such that $U^{x_0},\dots,U^{x_s}$ cover $\gamma$. Then $h_{x_i}=h_{x_{i+1}}$, $i=0,\dots,s-1$, as they agree on the overlaps, and hence $h_p=h_q$, as required.
\end{proof}
\medskip

\begin{proof}[Proof of Theorem~\ref{thm:semialgebraic}]
Let $R$ be a connected semianalytic set of pure dimension $d>0$, contained (resp. locally contained) in an irreducible (resp. locally irreducible) real analytic set $Z$ in $\C^n$ of the same dimension. Suppose first that $R$ is semialgebraic. Then there exists a real algebraic set $Y$ in $\C^n$, of pure dimension $d$, such that $R\subset Y$ and, for any open neighbourhood $U$ of any point $p\in\bar{R}$, $R\cap U$ contains an open subset of $Y$. Then, at every point $p\in\bar{R}$ where $Z_p$ is irreducible, we have $Z_p=Y_p$. It follows that, in the proof of Theorem~\ref{thm:holo-dim}, the complexification of $Z_p$ is a complex algebraic germ. Therefore it suffices to show that a projection from an algebraic set to a manifold is always Gabrielov regular. This however follows from Chevalley's theorem stating that a projection of an algebraic costructible set is itself constructible, and the fact that
\[
\dim_{\C}E=\dim_{\C}\overline{E}^{\mathrm{Zar}}
\]
for a constructible set $E$ (see~\cite[Ch.\,VII,\S\,8]{Lo}).
\smallskip

Suppose now that the dimension of the ambient space $\C^n$ is at most 2. Then, by the proof of Theorem~\ref{thm:holo-dim}, it suffices to show that a projection into $\C^n$ is always regular for $n\leq2$. For this, consider a mapping $\vp:A\to\C^2$ with $A$ an irreducible complex analytic set of dimension $d>0$, and let
\[
X=\{x\in A: \mathrm{fbd}_x>\lambda(\vp)\}
\]
be the locus of non-generic fibre dimension (cf.~Section~\ref{sec:background}). By the Cartan-Remmert Theorem (see \cite[Ch.\,V,\S\,3]{Lo}), $X$ is a proper analytic subset of $A$, and hence of dimension at most $d-1$. Moreover, the fibre dimension of $\vp$ is constant on $A\setminus X$, hence by the Remmert Rank Theorem (\cite[Ch.\,V,\S\,6]{Lo}), every point $x\in A\setminus X$ has arbitrarily small open neighbourhoods $V^x$ with $\vp(V^x)$ locally analytic in $\C^2$, of dimension $d-\lambda(\vp)$. On the other hand, since the generic fibre dimension of $\vp$ along every irreducible component of $X$ is at least $\lambda(\vp)+1$, it follows that
\[
\tag{*}
\dim_{\C}\vp(W)\leq(d-1)-(\lambda(\vp)+1)=(d-\lambda(\vp))-2
\]
for every subset $W\subset X$ (cf. $(\ddag)$ in Section~\ref{sec:proof-main}).
\par

Suppose that $\nR(\vp)\neq\varnothing$. Then $d-\lambda(\vp)\leq1$, for otherwise $\vp$ would be dominating; i.e., $r^1_x(\vp)=2$ for every $x$ in an open neighbourhood of $p$, and hence $\vp$ would have to be regular. But, clearly, $\nR(\vp)\subset X$ (by the Remmert Rank Theorem again), and hence, by $(*)$,
\[
\dim_{\C}\vp(\nR(\vp))\leq(d-\lambda(\vp))-2\leq-1\,.
\]
This is only possible if $\vp(\nR(\vp))=\varnothing$, hence $\nR(\vp)=\varnothing$; a contradiction.
\end{proof}

%%%%%%%%%%%%%%%%%%%%%%%%%%%%%%%%%%%%%%%%%%%%%%%%%%
\medskip
%%%%%%%%%%%%%%%%%%%%%%%%%%%%%%%%%%%%%%%%%%%%%%%%%%

\section{CR consequences}
\label{prop-proofs}

\begin{proof}[Proof of Proposition \ref{prop:cr}]
Let $M$ be a real analytic CR submanifold of $\C^n$ of dimension $d>0$ and
CR dimension $m\ge 0$. If we identify $T_p M$, the tangent space to $M$ at 
a point $p\in M$, with the appropriate real linear subspace of $\C^n$, then
clearly, the smallest complex linear subspace of $\C^n$ that contains $T_p M$ 
cannot be of dimension less than $d-m$. It follows then that no germ 
$M_p$ (for $p\in M$) can be contained in a complex analytic set of dimension less than $d-m$. 
Thus, to prove that $\dim_{HC}M_p=d-m$, it suffices to construct a germ of 
a complex analytic set of dimension  $d-m$ containing the germ $M_p$ for
every point $p\in M$. 

Without loss of generality assume that $p=0\in M$. Let 
$\phi: \R^d\to \C^n$ be a real analytic parametrization of $M$ near the origin 
with $\phi(0)=0$. Then each component $\phi_j$, $j=1,\dots,n$, of the map $\phi$ 
can be represented as a convergent power series 
$$
\phi_j(t)=\sum_{|\nu|>0}c_{j\nu} t^\nu,
$$
where $t\in\R^d$, $c_{j\nu}\in\C$, and $\nu\in\N^d$. Replacing $t$ by a complex 
variable $w=t+is$, we obtain a holomorphic map $\phi^c(w)$ from a neighbourhood $U$ 
of the origin in $\C^d$ into $\C^n$. 

Since $\phi^c$ is a holomorphic map, $d\phi^c(Jv)=Jd\phi^c(v)$ for any $v\in T_w\C^d$,
$w\in U$, where $J$ is the operator of multiplication by $i$. In particular, for 
$w\in U\cap\R^d$, the image of $d\phi^c(w)$ is precisely the smallest
complex subspace containing the images of vectors tangent to $\R^d$ at $w$.
It follows that $d\phi^c(w)$ has rank $d-m$ for $w\in\R^d\cap U$. 

Suppose that $B$ is any $(d-m+1)\times(d-m+1)$ submatrix of $d\phi^c$. Then $\det B$ is a
holomorphic function, and, since the rank of $d\phi^c$ equals $d-m$ on $\R^d\cap U$, it follows
that $\det B$ vanishes on $\R^d\cap U$. But then $\det B$ vanishes 
identically in $U$. Therefore, there is a neighbourhood $V$ of $U\cap \R^d$,
where $d\phi^c$ has constant rank $d-m$. By the Rank Theorem, $\phi^c(V)$ is a $(d-m)$-dimensional 
complex submanifold of $\C^n$ which contains $M$ (in a neighbourhood of $p$).
\end{proof}

\begin{proof}[Proof of Theorem \ref{cor:cr}]
Denote by $G$ the Grassmannian ${\rm Gr}(2n,d)$; i.e., the space of $d$-dimensional real linear
subspaces of $\R^{2n}\cong \C^n$, and let $G_m$ be the subset of $G$ consisting of CR subspaces
of $\C^n$ of CR dimension at least $m$. Let $L$ be a subspace of $\R^{2n}$ generated by $\{v^j\}$, 
$j=1,\dots,d$, $v^j=(v^j_{1},\dots,v^j_{2n})\in\R^{2n}$, and let $w^{j}=(w^j_1,\dots,w^j_{2n})=Jv^j$ (viewed as vectors
in $\C^n$). Then $L\in G_m$ if and only if, for 
any matrix of the form
\begin{equation*}
\begin{bmatrix}
v^1_1 & \cdots & v^d_1 & w^{1}_1 & \cdots & w^{d-2m+1}_1 \\
\vdots & & \vdots & \vdots & & \vdots\\
v^1_{2n} & \cdots & v^d_{2n} & w^{1}_{2n} & \cdots & w^{d-2m+1}_{2n}
\end{bmatrix},
\end{equation*}
any $(2d-2m+1)$-minor vanishes, whenever $2d-2m+1\leq2n$. 

If $M$ is a real analytic manifold in $\C^n$, then the assignment $p \mapsto T_pM$ defines a real
analytic map from $M$ into $G$. It follows that, for every $k\leq\left[\frac{d}{2}\right]$, the set 
\begin{equation}\label{e:ss}
S_k(M)=\{p\in M : \dim_{CR} T_p M\ge k\}
\end{equation}
is a real analytic subset of $M$ (possibly empty). Let $m=\min\{k: S_k(M)\ne\varnothing\}$. Then
$$M=S_m(M)\supset S_{m+1}(M)\supset S_{m+2}(M)\supset\cdots\supset S_{l}(M),\ \ 
l=\left[\frac{d}{2}\right].$$
By construction, $S_{m+1}(M)$ is a proper real analytic subset of $M$, and $M\setminus S_{m+1}(M)$ is a CR manifold of CR dimension $m$.

Suppose now that $R$ is an irreducible real analytic set of pure dimension $d$. By \cite{l1},
the sets $S_k(R^{\reg})$ in \eqref{e:ss} are semianalytic subsets of $R$. Let $S$ be the
set from Theorem~\ref{thm:holo-dim}. Then $M=R\setminus (R^{\sng}\cup S)$ is a real analytic manifold of dimension $d$. Let $M_1,M_2,\dots$ 
be the connected components of $M$. From the above considerations, each $M_j$ is a CR manifold outside a 
nowhere dense closed set. Further, by Theorem~\ref{thm:holo-dim}, the holomorphic closure dimension $h$ is 
constant on $M$, and therefore, it follows from Proposition~\ref{prop:cr} that each $M_j$ is a CR 
manifold of the same CR dimension $m=d-h$ for all~$j$. Hence, $R$ is a CR manifold in the complement 
of the semianalytic set $Y=S_{m+1}(R^{\reg})\cup R^{\sng}\cup S$.
\end{proof}

%%%%%%%%%%%%%%%%%%%%%%%%%%%%%%%%%%%%%%%%%%%%%%%%%%
\medskip
%%%%%%%%%%%%%%%%%%%%%%%%%%%%%%%%%%%%%%%%%%%%%%%%%%

\section{Examples}
\label{sec:examples}

In the following example we show that the exceptional set $S$ of Theorem~\ref{thm:holo-dim} may be non-empty.

\begin{example}
\label{ex:S-non-empty}
Let $R\subset\C^3_{\zeta}$ be the real analytic manifold defined by equations
\[
\begin{aligned}
\begin{cases}
x_2 =& x_1^2+y_1^2\\
y_2 =& 0\\
x_3 =& x_2 e^{x_1} \cos{y_1}\\
y_3 =& -x_2 e^{x_1} \sin{y_1}
\end{cases}
\quad\qquad\mathrm{or}\quad\qquad
\begin{cases}
\zeta_2 =& \zeta_1\bar{\zeta}_1\\
\bar{\zeta}_2 =& \bar{\zeta_1}\zeta_1\\
\zeta_3 =& \zeta_2e^{\bar{\zeta}_1}\\
\bar{\zeta}_3 =& \bar{\zeta}_2e^{\zeta_1}\ ,
\end{cases}
\end{aligned}
\]
where $\zeta_j=x_j+iy_j$, $j\leq3$. Then the complexification of $R_0$ equals the germ at $0\in\C^6_{(z,w)}$ of the set
\[
A=\{z_2-z_1w_1=w_2-z_1w_1=z_3-z_2e^{w_1}=w_3-w_2e^{z_1}=0\}\,,
\]
whose coordinate projection to $\C^3_z$ agrees with that of the set
\[
A'=\{(z,w)\in\C^6\,:\,w_2=w_3=z_2-z_1w_1=z_3-z_2e^{w_1}=0\}\,.
\]
Consider the local ring of $A'$ at the origin,
\begin{multline}
\notag
\OO_{A',0}=\frac{\C\{z_1,z_2,z_3,w_1,w_2,w_3\}}{(w_2,w_3,z_2-z_1w_1,z_3-z_2e^{w_1})} \cong\frac{\C\{z_1,z_2,z_3,w\}}{(z_2-z_1w,z_3-z_2e^{w})}\\
\cong\frac{\C\{z_1,z_2,z_3,v,w\}}{(z_1-v,z_2-vw,z_3-vwe^w)}\,.
\end{multline}
The latter is the local ring of the germ at the origin of the graph of the classical Osgood mapping
\[
\vp:\C^2\ni(v,w)\mapsto(v,vw,vwe^w)\in\C^3_z\,.
\]
Now, by~\cite[\S\,2.5]{GR}, every power series $F\in\C\{z_1,z_2,z_3\}$ satisfying $F(v,vw,vwe^w)=0$, vanishes identically. It follows that $\ker\vp^*_0=\{0\}$, and hence $r^3_0(\vp)=3$, whilst $r^1_0(\vp)=2$. Therefore, $\vp$ is not regular at the origin, and consequently, the set $S$ from Theorem~\ref{thm:holo-dim} is non-empty. In fact, $S=\{0\}$.
\end{example}

\begin{remark}
\label{rem:sky's-the-limit}
As mentioned in Section~\ref{sec:intro}, in general, the holomorphic closure dimension of a real analytic set $R$ in $\C^n$ is bounded only by the dimension $n$ of the ambient space. Indeed, by~\cite[\S\,2.5]{GR} again, there exists, for every $n\geq3$, a monomorphism of $\C$-algebras
\[
\theta:\C\{z_1,\dots,z_n\} \hookrightarrow \C\{v,w\}\,,
\]
and hence, a holomorphic mapping $\vp:\C^2\to\C^n$ with $\vp^*_0=\theta$. It follows that $\ker\vp^*_0=\{0\}$, and hence $r^3_0(\vp)=\dim\C\{z_1,\dots,z_n\}/(0)=n$. Of course, $r^1_0(\vp)\leq2$, since, in any case, this rank is bounded above by the dimension of the source.
Then, as in the above example, we may construct a $2$-dimensional real analytic manifold $R$ in $\C^n_{\zeta}$, such that the complexification $R^c_0$ has the same projection to $\C^n_z$ as the graph of the map $\vp$, and consequently $\dim_{HC}R_0=n$.
\end{remark}
\medskip

One might suppose that the holomorphic closure dimension of a real analytic set $R$ admits some sort of stratification. It turns out that this is not the case. We are indebted to E. Bierstone for suggesting the following example, which shows that no semianalytic (even subanalytic) stratification for $\dim_{HC}R_p$ exists beyond $S$ in general.

\begin{example}
\label{ex:no-semi-coherent}
In~\cite{P1} (see also~\cite[Ex.\,1.29]{BM2}), Paw{\l}ucki constructed an analytic mapping
\[
(v,x,y)\stackrel{\vp}{\mapsto}(v,x,xy,xg(v,y),xh(v,x,y))
\]
from a small open neighbourhood of the origin in $\C^3$ to $\C^5$, with the following properties:
\begin{itemize}
\item[(i)] $\nR(\vp)=\{x=y=0\}$
\item[(ii)] $r^2_a(\vp)=5$ at every point $a\in\{(0,0,0)\}\cup\{(1/n,0,0):n\geq1\}$
\item[(iii)] $r^3_a(\vp)=4$ at every other $a\in\nR(\vp)$\,.
\end{itemize}
Along the lines of the previous example, we construct the set $R\subset\C^5_{\zeta}$ as defined by equations
\[
\begin{cases}
\zeta_3 =& \zeta_2\bar{\zeta}_2\\
\bar{\zeta}_3 =& \bar{\zeta_2}\zeta_2\\
\zeta_4 =& \zeta_2g(\zeta_1,\bar{\zeta}_2)\\
\bar{\zeta}_4 =& \bar{\zeta}_2\bar{g}(\bar{\zeta}_1,\zeta_2)\\
\zeta_5 =& \zeta_2h(\zeta_1,\zeta_2,\bar{\zeta}_2)\\
\bar{\zeta}_5 =& \bar{\zeta}_2\bar{h}(\bar{\zeta}_1,\bar{\zeta}_2,\zeta_2)\,.
\end{cases}
\]
Then
\[
\dim_{HC}R_p=
\begin{cases}
3 &: p\in R\setminus S\\
4 &: p\in S\setminus T\\
5 &: p\in T\,,
\end{cases}
\]
where \ $S=\{\zeta_2=\zeta_3=\zeta_4=\zeta_5=0\}$, \ and \ $T=\{\zeta\in S:\zeta_1=0\mathrm{\ or}\ \zeta_1=1/n, n\geq1\}$\, is not subanalytic.
\end{example}

\begin{example}
\label{ex:cartan}
As suggested by the proof of Theorem~\ref{thm:holo-dim}, the set $S$ of the exceptional holomorphic closure dimension may have non-empty interior in an irreducible real analytic $R$, if $R$ is not of pure dimension. Consider, for instance, the Cartan umbrella
\[
R=\{(z_1,z_2)=(x_1+iy_1,x_2+iy_2)\in\C^2\ :\ x_2(x_1^2+y_1^2)-x_1^3=y_2=0\}\,,
\]
which is of real dimension $2$ except at the points of the ``stick'' $S=\{x_1=y_1=y_2=0\}\setminus\{0\}$. Then $\dim_{HC}R_p=2$ for every $p\in R\setminus S$, and $\dim_{HC}R_p=1$ for every $p\in S$, since then $R_p$ is contained in $X_p=(\{z_1=0\})_p$.
\end{example}

%%%%%%%%%%%%%%%%%%%%%%%%%%%%%%%%%%%%%%%%%%%%%%%%%%
\medskip
%%%%%%%%%%%%%%%%%%%%%%%%%%%%%%%%%%%%%%%%%%%%%%%%%%
%References
%%%%%%%%%%%%%%%%%%%%%%%%%%%%%%%%%%%%%%%%%%%%%%%%%%
\bibliographystyle{amsplain}

\begin{thebibliography}{99}

\bibitem {ABM} J. Adamus, E. Bierstone and P. D. Milman,
\textit{Uniform linear bound in Chevalley's lemma}, Canad. J. Math. \textbf{60} (2008), no.\,4, 721--733.

\bibitem {BM} E. Bierstone, P. D. Milman,
\textit{Semianalytic and subanalytic sets}, Inst. Hautes {\'E}tudes Sci. Publ. Math. \textbf{67} (1988), 5--42.

\bibitem {BM3} E. Bierstone, P. D. Milman,
\textit{Canonical desingularization in characteristic zero by blowing up the maximum strata of a local invariant}, Invent. Math. \textbf{128}  (1997), no.\,2, 207--302.

\bibitem {BM2} E. Bierstone, P. D. Milman,
\textit{Geometric and differential properties of subanalytic sets}, Ann. of Math. (2) \textbf{147} (1998), 731--785.

\bibitem{C} H. Cartan,
\textit{Vari{\'e}t{\'e}s analytiques r{\'e}elles et vari{\'e}tes analytiques complexes}, Bull. Soc. Math. France \textbf{85} (1957), 77--99.

\bibitem {Ga} A. Gabrielov,
\textit{Formal relations between analytic functions}, Math. USSR Izviestija \textbf{7} (1973), 1056--1088.

\bibitem {GR} H. Grauert, R. Remmert,
``Analytische Stellenalgebren'', Springer Verlag, New York, 1971.

\bibitem{hl} R. Harvey and B. Lawson,
\textit{On boundaries of complex analytic varieties. I}, Ann. of Math. (2) \textbf{102} (1975), no.\,2, 223--290.

\bibitem{L2} S. {\L}ojasiewicz,
``Ensembles semi-analytiques'', Inst. Hautes {\'E}tudes Sci., Bures-sur-Yvette, 1964.

\bibitem{l1} S. {\L}ojasiewicz,
\textit{Sur la semi-analycit\'e des images inverses par l'application-tangente}, Bull. Acad. Polon. Sci. S\'er. Sci. Math. {\bf 27} (1979), no.\,7-8, 525--527 (1980). 

\bibitem {Lo} S. {\L}ojasiewicz,
``Introduction to Complex Analytic Geometry'', Birkh\"{a}user, Basel, 1991.

\bibitem{n} R. Narasimhan,
``Introduction to the Theory of Analytic Spaces'', Lecture Notes in Mathematics, Vol.\,25, Springer, 1966. 

\bibitem {P1} W. Paw{\l}ucki,
\textit{On relations among analytic functions and geometry of subanalytic sets}, Bull. Polish Acad. Sci. Math. \textbf{37} (1989), 117--125.

\bibitem {P2} W. Paw{\l}ucki,
\textit{On Gabrielov's regularity condition for analytic mappings}, Duke Math. J. \textbf{65} (1992), 299--311.

\bibitem {S} R. Shafikov,
\textit{Real analytic sets in complex spaces and CR maps}, Math. Z. \textbf{256} (2007), 757--767.

\end{thebibliography}

\end{document}